\documentclass{amsart}%
\usepackage{amsfonts}
\usepackage{amsmath}
\usepackage{amssymb}
\usepackage{graphicx}%
\providecommand{\U}[1]{\protect\rule{.1in}{.1in}}
\providecommand{\U}[1]{\protect\rule{.1in}{.1in}}
\newtheorem{theorem}{Theorem}
\theoremstyle{plain}
\newtheorem{acknowledgement}{Acknowledgement}

\newtheorem{corollary}{Corollary}

\newtheorem{example}{Example}

\newtheorem{lemma}{Lemma}

\newtheorem{proposition}{Proposition}
\newtheorem{remark}{Remark}

\numberwithin{equation}{section}
\begin{document}
\title[ ]{Homeomorphisms generated from overlapping affine iterated function systems}
\author{Michael F. Barnsley}
\address{Department of Mathematics\\
Australian National University\\
Canberra, ACT, Australia\\
 }
\author{Brendan Harding}
\address{Department of Mathematics\\
Australian National University\\
Canberra, ACT, Australia\\
 }
\author{Andrew Vince}
\address{Department of Mathematics\\
University of Florida\\
Gainesville, FL 32611-8105, USA\\
 }
\email{michael.barnsley@maths.anu.edu.au,}
\urladdr{http://www.superfractals.com{}}

\begin{abstract}
We develop the theory of fractal homeomorphism generated from pairs
overlapping affine iterated function systems.

\end{abstract}
\maketitle

\section{\label{introsec}Introduction}

We consider a pair of dynamical systems, $W:[0,1]\rightarrow\lbrack0,1]$ and
$L:[0,1]\rightarrow\lbrack0,1],$ as illustrated in Figure \ref{figaug17}. $W$
is differentiable on both $[0,\rho]$ and $(\rho,1]$, with slope greater than
$1/\lambda>1,$ and $L$ is piecewise linear with slope $1/\gamma>1$. If
$h(W)=-\ln$ $\gamma$ is the topological entropy of $W$ then there is
$p\in(0,1)$ such that the two systems are topologically conjugate, i.e. there
is a homeomorphism $H:[0,1]\rightarrow\lbrack0,1]$ such that $W=HLH^{-1}$.
This follows from \cite[Theorem 1]{denker}. It can also be deduced from
\cite{parry}. What is not known, prior to this work, is the explicit
relationship between $W$, on the left in Figure \ref{figaug17}$,$ and the
parameters $p$ and $\gamma$ that uniquely define $L$, on the right in Figure
\ref{figaug17}.%

\begin{figure}[ptb]%
\centering
\includegraphics[
height=2.5598in,
width=5.5486in
]%
{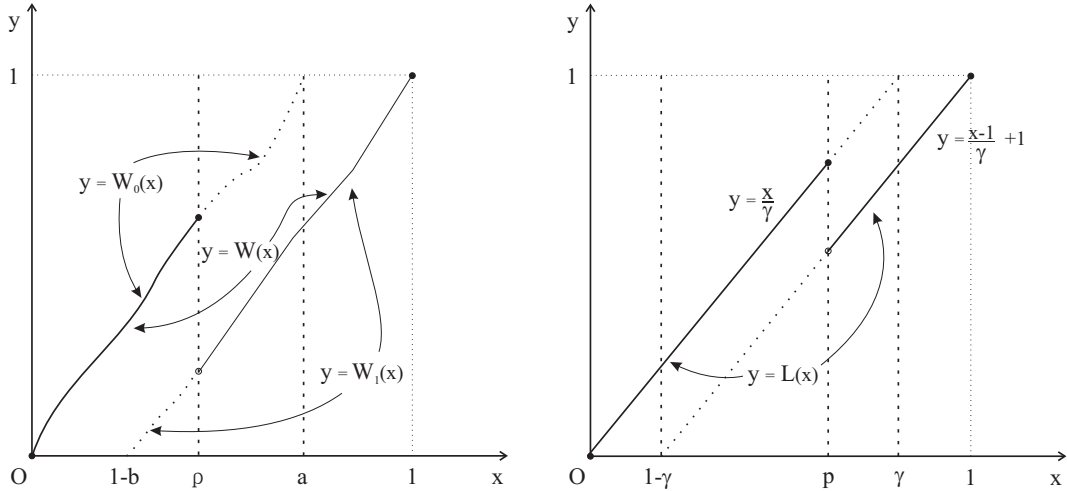}%
\caption{Theorem \ref{theorem1} provides a conditions under which the two
systems illustrated here are topologically conjugate. Theorem
\ref{theorem2ndpart} provides a geometrical characterization of the
conjugacy.}%
\label{figaug17}%
\end{figure}

In this paper we prove constructively the existence of $L$ and establish
analytic expressions, that use only two itineraries of $W$, from which both
topological invariants $\gamma$ and $p$ can be deduced. We also provide a
direct construction for the graph of $H$. While $\gamma$ has been much
studied, the parameter $p$ is also of interest because it measures the
asymmetry of the set of itineraries of $W$. Indeed, one motivation is our
desire to establish, and to be able to compute, fractal homeomorphisms between
attractors of various overlapping iterated function systems, as explained in
Section \ref{masksec}, for applications such as those in \cite{BHI}. Our
approach is of a constructive character, similar to that in \cite{milnor}, but
founded in the theory of overlapping iterated function systems and associated
families of discontinuous dynamical systems. We make use of an analogue of the
kneading determinant of \cite{milnor}, appropriate for discontinuous interval
maps, and thereby avoid measure-theoretic existential demonstrations such as
those in \cite{parry, hoffbauer}.

Let $I=\{0,1\}$ and $I^{\infty}=\{0,1\}\times\{0,1\}\times...$ with the
product topology. Each point $x\in\lbrack0,1]$ has a unique itinerary
$\tau(x)\in I,$ where the $k^{th}$ component of $\tau(x)$, denoted by
$\tau(x)_{k}$, equals $0$ or $1$ according as $W^{k}(x)\in\lbrack0,\rho],$ or
$(\rho,1],$ respectively, for all $k\in\mathbb{N}$. Corresponding to each
$x\in\lbrack0,1)$ we associate an analytic function%
\[
\tau(x)(\zeta):=(1-\zeta)\sum\limits_{k=0}^{\infty}\tau(x)_{k}\zeta^{k}\text{,
}\zeta\in\mathbb{C}\text{, }\left\vert \zeta\right\vert <1.
\]
Our first main result specifies the invariants $p$ and $\gamma$ in terms of
two of these functions, and describes the homeomorphism $H$.

\begin{theorem}
\label{theorem1} The topological entropy of $W$ is $-\ln\gamma$ where $\gamma$
is the unique solution of%
\begin{equation}
\tau(\rho)(\gamma)=\tau(\rho+)(\gamma),\text{ with }\tau(\rho)(\varsigma
)<\tau(\rho+)(\varsigma)\text{ for }\varsigma\in\lbrack0.5,\gamma)\text{,}
\label{innovationeq}%
\end{equation}
and
\[
p=\tau(\rho)(\gamma).
\]
Moreover,
\[
H(x)=\tau(x)(\gamma)\text{, for all }x\in\lbrack0,1),\text{ and }%
H(1)=1\text{.}%
\]

\end{theorem}

Here $\tau(\rho+)=\underset{\varepsilon\rightarrow0}{\lim}$ $\tau
(\rho+\left\vert \varepsilon\right\vert )$. Our second main result provides a
geometrical construction of the homeomophism $H.$ Let
\[
\square=\{(x,y)\in\mathbb{R}^{2}:0\leq x\leq1,0\leq y\leq1\},
\]
and let $\mathbb{H}$ denote the nonempty compact subsets of $\square$ with the
Hausdorff topology.

\begin{theorem}
\label{theorem2ndpart} If $gr(H)$ is the graph of $H$ then%
\[
gr(H)=\bigcap\limits_{k\in\mathbb{N}}r^{k}(\square)=\lim_{k\rightarrow\infty
}r^{k}(\square)
\]
where
\[
r:\mathbb{H\rightarrow H\ni}S\mapsto F_{0}(S\cap P)\cup F_{1}(S\cap Q),
\]%
\[
F_{0}:\square\rightarrow\square\mathbb{\ni}(x,y)\mapsto(\gamma x,W_{0}%
^{-1}(y)),F_{1}:\square\rightarrow\square\mathbb{\ni}(x,y)\mapsto(\gamma
x+1-\gamma,W_{1}^{-1}(y)),
\]%
\[
P=\{(x,y):x\leq p/\gamma,y\leq W_{0}(\rho)\}\text{, }Q=\{(x,y):x\geq
p/\gamma+1-1/\gamma,y\geq W_{1}(\rho)\}.
\]

\end{theorem}

The expression $gr(H)=\bigcap\limits_{k\in\mathbb{N}}r^{k}(\square
)=\lim_{k\rightarrow\infty}r^{k}(\square)$ is a localized version of the
expression $A=\lim_{k\rightarrow\infty}\mathcal{F}^{k}(\square)$ for the
attractor $A$ of a hyperbolic iterated function system $\mathcal{F}$ on
$\square$.

This paper is, in part, a continuation of \cite{mihalache}. In
\cite{mihalache} we analyse in some detail the topology and structure of the
address space/set of itineraries associated with $W$ and $W_{+}$. In this
paper we recall and use key results from \cite{mihalache}; but here the point
of view is that of masked iterated function systems, whereas in
\cite{mihalache} the point of view is classical symbolic dynamics. The novel
innovation in this paper is the introduction and exploitation of the family of
analytic functions in equation (\ref{innovationeq}), yielding Theorem
\ref{theorem1}.

** Special case : the affine case; implicit function theorem gives the
dependence of p on a,b,and rho.

** Outline of sections and their contents, with focus on how the proofs work.

** Comments on related work by Konstantin Igudesman \cite{igudesman1,
igudesman2}.

\section{\label{backgroundsec}Background and Notation}

\subsection{\label{IFSsec}Iterated functions systems, their attractors, coding
maps, sections and address spaces}

Let $\mathbb{X}$ be a complete metric space. Let $f_{i}:\mathbb{X\rightarrow
}\mathbb{X}$ ($i=0,1$) be contraction mappings. Let $\mathbb{H=H(X)}$ be the
nonempty compact subsets of $\mathbb{X}$. Endow $\mathbb{H}$ with the
Hausdorff metric. We use the same symbol $\mathcal{F}$ for the hyperbolic
iterated function system $\mathcal{(}\mathbb{X};f_{0},f_{1}),$ for the set of
maps $\{f_{0},f_{1}\}$, and for the contraction mapping%
\[
\mathcal{F}:\mathbb{H\rightarrow H}\text{, }S\mapsto f_{0}(S)\cup
f_{1}(S)\text{.}%
\]
Let $A\in\mathbb{H}$ be the fixed point of $\mathcal{F}$. We refer to $A$ as
the \textit{attractor} of $\mathcal{F}$.

Let $I=\{0,1\}$ and let $I^{\infty}=\{0,1\}\times\{0,1\}\times...$ have the
product topology induced from the discrete topology on $I$. For $\sigma\in
I^{\infty}$ we write $\sigma=\sigma_{0}\sigma_{1}\sigma_{2}\ldots,$ where
$\sigma_{k}\in I$ for all $k\in\mathbb{N}$. The product topology on
$I^{\infty}$ is the same as the topology induced by the metric $d(\omega
,\sigma)=2^{-k}$ where $k$ is the least index such that $\omega_{k}\neq
\sigma_{k}$. It is well known that $(I^{\infty},d)$ is a compact metric space.
For $\sigma\in I^{\infty}$ and $n\in\mathbb{N}$ we write $\sigma|_{n}%
=\sigma_{0}\sigma_{1}\sigma_{2}...\sigma_{n}$. The \textit{coding map} for
$\mathcal{F}$ is
\[
\pi:I^{\infty}\rightarrow A\text{, }\sigma\mapsto\lim_{k\rightarrow\infty
}f_{\sigma|_{k}}(x),
\]
where $x\in\mathbb{X}$ is fixed and $f_{\sigma|_{k}}(x)=f_{\sigma_{0}}\circ
f_{\sigma_{1}}\circ...\circ f_{\sigma_{k}}(x).$ The map $\pi:I^{\infty
}\rightarrow A$ is a continuous surjection, independent of $x$. We refer to an
element of $\pi^{-1}(x)$ as an \textit{address} of $x\in A$. A
\textit{section} for $\mathcal{F}$ is a map $\tau:A\rightarrow I^{\infty}$
such that $\pi\circ\tau=i_{A}$, the identity map on $A$. We also say that
$\tau$ is a section \textit{of} $\pi$. We refer to $\Omega=\tau(A)$ as an
\textit{address space} for $A$ (associated with $\mathcal{F}$) because
$\Omega$ is a subset of $I^{\infty}$ and it is in bijective correspondence
with $A$.

We write $\overline{E}$ to denote the closure of a set $E$. But we write
$\overline{0}=000...,\overline{1}=111...\in I^{\infty}$. For $\sigma
=\sigma_{0}\sigma_{1}\sigma_{2}\ldots\in I^{\infty}$ we write $0\sigma$ to
mean $0\sigma_{0}\sigma_{1}\sigma_{2}\ldots\in I^{\infty}$ and $1\sigma
=0\sigma_{0}\sigma_{1}\sigma_{2}\ldots\in I^{\infty}$.

\subsection{\label{ordersec}Order relation on code space, top sections and
shift invariance}

We define a total order relation $\preceq$ on $I^{\infty},$ and on $I^{n}$ for
any $n\in\mathbb{N}$, by $\sigma\prec\omega$ if $\sigma\neq\omega$ and
$\sigma_{k}<\omega_{k}$ where $k$ is the least index such that $\sigma_{k}%
\neq\omega_{k}$. For $\sigma,\omega\in I^{\infty}$ with $\sigma\preceq\omega$
we define%
\begin{align*}
\lbrack\sigma,\omega]  &  :=\{\zeta\in I^{\infty}:\sigma\preceq\zeta
\preceq\omega\},(\sigma,\omega):=\{\zeta\in I^{\infty}:\sigma\prec\zeta
\prec\omega\},\\
(\sigma,\omega]  &  :=\{\zeta\in I^{\infty}:\sigma\prec\zeta\preceq
\omega\},[\sigma,\omega):=\{\zeta\in I^{\infty}:\sigma\preceq\zeta\prec
\omega\}.
\end{align*}

It is helpful to note the following alternative characterization of the order
relation $\preceq$ on $I^{\infty}.$ Since the standard Cantor set
$C\subset\lbrack0,1]\subset\mathbb{R}$ is totally disconnected, and is the
attractor of the iterated function system $([0,1];f_{0}(x)=x/3,f_{1}%
(x)=x/3+2/3)$, the coding map $\pi_{C}:I^{\infty}\rightarrow C,\sigma
\mapsto\sum\limits_{k=0}^{\infty}2\sigma_{k}/3^{k+1},$ is a homeomorphism. The
order relation $\preceq$ on $I^{\infty}$ can equivalently be defined by
$\sigma\prec\omega$ if and only if $\pi_{C}(\sigma)<\pi_{C}(\omega)$.

The order relation $\prec$ on $I^{\infty}$ can be used to define the
corresponding \textit{top section }$\mathcal{\tau}_{top}:A\rightarrow
I^{\infty}$ for $\mathcal{F}$, according to
\[
\mathcal{\tau}_{top}(x)=\max\pi^{-1}(x)\text{.}%
\]
Top sections are discussed in \cite{monthly}. Let $\Omega_{top}=\mathcal{\tau
}_{top}(A)$ and let $S:I^{\infty}\rightarrow I^{\infty}$ denote the left-shift
map $\sigma_{0}\sigma_{1}\sigma_{2}...\mapsto\sigma_{1}\sigma_{2}\sigma
_{3}...$. We have
\[
S(\Omega_{top})\subseteq\Omega_{top}%
\]
with equality when $f_{1}$ is injective, \cite[Theorem 2]{monthly}.

We say that a \textit{section} $\tau$ \textit{is shift invariant} when
$S(\Omega)=\Omega$, and \textit{shift-forward invariant when }$S(\Omega
)\subset\Omega.$ The examples considered later in this paper involve shift
invariant sections.

The branches of $S^{-1}$ are $s_{i}:I^{\infty}\rightarrow I^{\infty}$ with
$s_{i}(\sigma)=i\sigma$ $(i=0,1).$ Both $s_{0}$ and $s_{1}$ are contractions
with contractivity $1/2$. $I^{\infty}$ is the attractor of the iterated
function system $(I^{\infty};s_{0},s_{1})$. We write $2^{I^{\infty}}$ to
denote the set of all subsets of $I^{\infty}.$

\subsection{\label{masksec}Masks, masked dynamical systems and masked
sections}

Sections are related to masks. A \textit{mask} $\mathcal{M}$ \textit{for}
$\mathcal{F}$ is a pair of sets$,$ $M_{i}\subset f_{i}(A)$ $(i=0,1)$, such
that $M_{0}\cup M_{1}=A$ and $M_{0}\cap M_{1}=\emptyset.$ If the maps
$f_{i}|_{A}:A\rightarrow A$ $(i=0,1)$ are invertible, then we define a
\textit{masked dynamical system} \textit{for} $\mathcal{F}$ to be
\[
W_{\mathcal{M}}:A\rightarrow A,\text{ }M_{i}\ni x\mapsto f_{i}^{-1}(x),\text{
}(i=0,1).
\]
It is proved in \cite[Theorem 4.3]{BHI} that, given a mask $\mathcal{M},$ if
the maps $f_{i}|_{A}:A\rightarrow A$ $(i=0,1)$ are invertible, we can define a
section for $\mathcal{F}$, that we call a \textit{masked section}
$\mathcal{\tau}_{\mathcal{M}}$ for $\mathcal{F}$, by using itineraries of
$W_{\mathcal{M}}$, as follows. Let $x\in A$ and let $\left\{  x_{n}\right\}
_{n=0}^{\infty}$ be the orbit of $x$ under $W_{\mathcal{M}}$; that is,
$x_{0}=x$ and $x_{n}=W_{\mathcal{M}}^{n}(x_{0})$ for $n=1,2,...$. Define
\begin{equation}
\mathcal{\tau}_{\mathcal{M}}(x)=\sigma_{0}\sigma_{1}\sigma_{2}...
\label{itineraryeq}%
\end{equation}
where $\sigma_{n}\in I$ is the unique symbol such that $x_{n}\in M_{\sigma
_{n}}$for all $n\in\mathbb{N}$.

Sections defined using itineraries of masked dynamical systems are shift invariant.

\begin{proposition}
\label{maskprop} Let the maps $f_{i}|_{A}:A\rightarrow A$ $($ $i=1,2)$ be invertible.

(i) Any mask $\mathcal{M}$ for $\mathcal{F}$ defines a shift-forward invariant
section, $\tau_{\mathcal{M}}:A\rightarrow I^{\infty}$, for $\mathcal{F}$.

(ii) Let $\Omega_{\mathcal{M}}=\tau_{\mathcal{M}}(A)$. The following diagram
commutes:%
\[%
\begin{array}
[c]{ccc}%
\Omega_{\mathcal{M}} & \overset{S|_{\Omega_{\mathcal{M}}}}{\rightarrow} &
\Omega_{\mathcal{M}}\\
\pi\downarrow\uparrow\tau_{\mathcal{M}} &  & \pi\downarrow\uparrow
\tau_{\mathcal{M}}\\
A & \underset{W_{\mathcal{M}}}{\rightarrow} & A
\end{array}
.
\]

(iii)\ Any section $\tau:A\rightarrow I^{\infty}$ for $\mathcal{F}$ defines a
mask $\mathcal{M}_{\tau}$ for $\mathcal{F}.$

(iv) If the section $\tau$ in (iii) is shift-forward invariant then $\tau
=\tau_{\mathcal{M}_{\tau}}.$
\end{proposition}

\begin{proof}
(i) Compare with \cite[Theorem 4.3]{BHI}. If the maps are invertible, we can
use $\mathcal{M}$ to define an itinerary for each $x\in$ $A$, as in
(\ref{itineraryeq}), yielding a section $\tau_{\mathcal{M}}$ for $\mathcal{F}%
$. By construction, $\tau_{\mathcal{M}}$ is shift-forward invariant. (ii) We
show that $\tau_{\mathcal{M}}W_{\mathcal{M}}\pi\sigma=S\sigma$ for all
$\sigma\in\Omega_{\mathcal{M}}$. We have $\pi\sigma$ is a point $x\in A$ that
possesses address $\sigma\in\Omega_{\mathcal{M}}$. But $W_{\mathcal{M}}$ acts
by applying $f_{\sigma_{0}}^{-1}$ to $x=f_{\sigma_{0}}\circ f_{\sigma_{1}%
}\circ f_{\sigma_{2}}...$ yielding the point $W_{\mathcal{M}}\pi
\sigma=f_{\sigma_{1}}\circ f_{\sigma_{2}}...$ which tells us that $\sigma
_{1}\sigma_{2}\sigma_{3}..$ is \textit{an} address of $W_{\mathcal{M}}%
\pi\sigma$. But since $S\sigma\in\Omega_{\mathcal{M}}$ this address must be
the unique address in $\Omega_{\mathcal{M}}$ of $W_{\mathcal{M}}\pi\sigma.$ It
follows that $W_{\mathcal{M}}\pi\sigma=\sigma_{1}\sigma_{2}\sigma
_{3}..=S\sigma$. (iii) Given a section $\tau:A\rightarrow I^{\infty},$ we
define a mask $\mathcal{M}_{\tau}$ by $M_{i}=\{x\in A:\tau(x)_{0}=i\}(i=1,2).$
(iv) This is essentially the same as the proof of (ii).
\end{proof}

\subsection{\label{transformsec}Fractal transformations}

Let $\mathcal{G}=(\mathbb{Y};g_{1},g_{2})$ be a hyperbolic iterated function
system, with attractor $A_{\mathcal{G}}$ and coding map $\pi_{\mathcal{G}}$.
We refer to any mapping of the form
\[
\mathcal{T}_{\mathcal{FG}}:A\rightarrow A_{\mathcal{G}},x\mapsto
\pi_{\mathcal{G}}\circ\tau(x)\text{,}%
\]
where $\tau$ is a section of $\mathcal{F}$, as a \textit{fractal
transformation. }Later in this paper we construct and study fractal
transformations associated with certain overlapping iterated function systems,
such as those suggested by the left-hand panel in Figure \ref{fig-trans}. We
will use part (iv) of the following result to establish Theorem \ref{theorem1}.

\begin{proposition}
[ \cite{BHI}]\label{basicconjugacythm} Let $\tau:A\rightarrow I^{\infty}$ be a
section for $\mathcal{F}$ and let $\Omega=\tau(A)$ be an address space the
attractor $A$ of $\mathcal{F}$.

(i) If $\Omega$ is an address space for $\mathcal{G}$ then $\mathcal{T}%
_{\mathcal{FG}}:A\rightarrow A_{\mathcal{G}}$ is a bijection.

(ii) If, whenever $\sigma,\omega\in\overline{\Omega}$, $\pi(\sigma)=\pi
(\omega)\Rightarrow$ $\pi_{\mathcal{G}}(\sigma)=\pi_{\mathcal{G}}(\omega)$,
then $\mathcal{T}_{\mathcal{FG}}:A\rightarrow A_{\mathcal{G}}$ is continuous.

(iii) If, whenever $\sigma,\omega\in\overline{\Omega}$, $\pi(\sigma
)=\pi(\omega)\Leftrightarrow$ $\pi_{\mathcal{G}}(\sigma)=\pi_{\mathcal{G}%
}(\omega)$, then $\mathcal{T}_{\mathcal{FG}}:A\rightarrow A_{\mathcal{G}}$ is
a homeomorphism.

(iv) If $\tau$ is a masked section of $\mathcal{F}$ such that the condition in
(iii) holds then the corresponding pair of masked dynamical systems,
$W_{\mathcal{M}}:A\rightarrow A$ and, say, $W_{\mathcal{M}_{\mathcal{G}}%
}:A_{\mathcal{G}}\rightarrow A_{\mathcal{G}}$ are topologically conjugate.
\end{proposition}

Here $W_{\mathcal{M}_{\mathcal{G}}}$ is defined in the obvious way, as
follows. Since $\Omega$ is an address space for the attractor $A_{\mathcal{G}%
}$ of $\mathcal{G}$ and it is also shift-forward invariant -since it is a
masked address space for $\mathcal{F}$- it defines a shift-forward invariant
section $\mathcal{\tau}_{\mathcal{G}}$ for $\mathcal{G}$, so by Proposition
\ref{maskprop}(iii), it defines a mask $\mathcal{M}_{\mathcal{G}}$ for
$\mathcal{G}$ such that $\mathcal{\tau}_{\mathcal{G}}=\mathcal{\tau
}_{\mathcal{M}_{\mathcal{G}}}$; we use this latter mask to define the masked
dynamical system $W_{\mathcal{M}_{\mathcal{G}}}:A_{\mathcal{G}}\rightarrow
A_{\mathcal{G}}$.

\begin{proof}
(i) follows at once from the fact that $\Omega$ is a section for both
$\mathcal{F}$ and $\mathcal{G}$. (ii) and (iii) are proved in \cite[Theorem
3.4]{BHI}. (iv) is immediate, based on the definition of $W_{\mathcal{M}%
_{\mathcal{G}}}:A_{\mathcal{G}}\rightarrow A_{\mathcal{G}}$.
\end{proof}

\begin{remark}
All of the results in Section \ref{backgroundsec} apply to any hyperbolic
iterated function systems of the form $\mathcal{F}=(\mathbb{X};f_{1}%
,f_{2},...,f_{N})$ where $N$ is an arbitrary finite positive integer.
\end{remark}

\section{\label{overlapifssec}Overlapping iterated function systems of two
monotone increasing interval maps}

\subsection{\label{generalstructuresec}General structure}

Here we consider iterated function systems, related to $W$ as introduced at
the start of Section \ref{introsec}, that involve overlapping monotone
increasing interval maps. We introduce two families of masks and characterize
the associated sections and address spaces.

Let $\mathbb{X}$ be $[0,1]\subset\mathbb{R}$ with the Euclidean metric. Let
$0<\lambda<1.$ Let
\begin{equation}
\mathcal{F}=([0,1]\subset\mathbb{R};f_{0}(x),f_{1}(x))\label{ifsequation}%
\end{equation}
where $f_{i}(i)=i,0<f_{i}(y)-f_{i}(x)<\lambda(y-x)$ for all $x<y$, $(i=0,1).$
Both maps are monotone strictly increasing contractions. We also require
$f_{0}(1)=a\geq1-b=f_{1}(0)$ with $0<a,b<1$. See Figure \ref{fig-trans}. The
attractor of $\mathcal{F}$ is $A=[0,1]=f_{0}([0,1])\cup f_{1}([0,1])$ and the
coding map is $\pi:I^{\infty}\rightarrow\lbrack0,1]$.%

\begin{figure}[ptb]%
\centering
\includegraphics[
height=2.6844in,
width=5.5486in
]%
{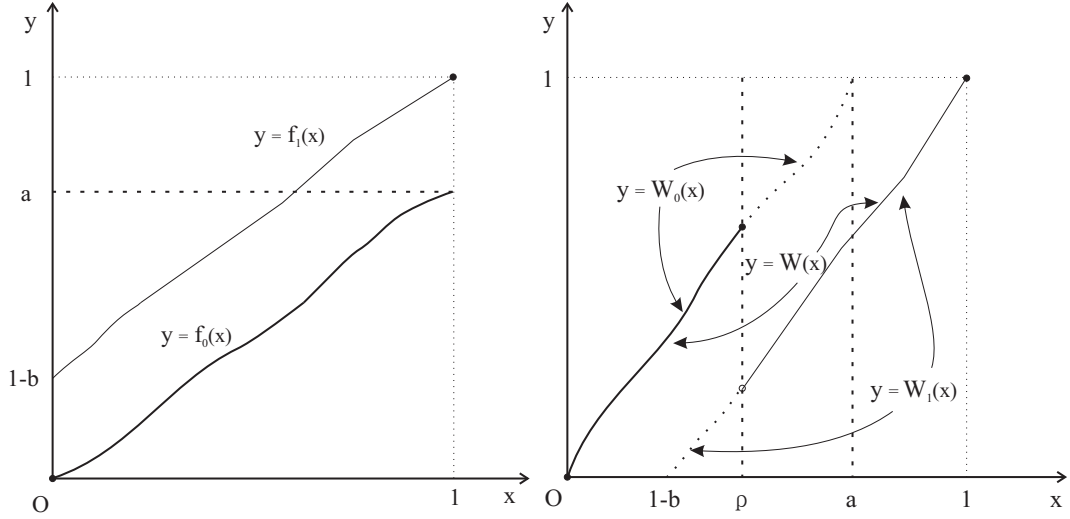}%
\caption{Left: the graphs of two functions that comprise an iterated function
system $\mathcal{F}$, as described in the text. Right: the graph of a masked
dynamical system for $\mathcal{F}$.}%
\label{fig-trans}%
\end{figure}

We define a one parameter family of masks for $\mathcal{F}$,%
\[
\{\mathcal{M}_{\rho}:0<1-b\leq\rho\leq a<1\},
\]
by $M_{0}=[0,\rho],$ $M_{1}=(\rho,1]$. The corresponding masked dynamical
system is
\begin{equation}
W:[0,1]\rightarrow\lbrack0,1]\ni x\mapsto\left\{
\begin{array}
[c]{c}%
f_{0}^{-1}(x)\text{ if }x\in\lbrack0,\rho],\\
f_{1}^{-1}(x)\text{ otherwise},
\end{array}
\right.  \label{Wequation}%
\end{equation}
as graphed in Figure \ref{fig-trans}.

The masked section for $\mathcal{F}$ is $\tau=\tau(\rho),$ and the masked code
space is $\Omega=\Omega(\rho)$. The dependence on $\rho$ is implicit except
where we need to draw attention to it. For convenient reference we note that,
for all $x\in\lbrack0,1]$,%
\[
\tau(x)=\sigma_{0}\sigma_{1}\sigma_{2}...\in I^{\infty}\text{ \textit{where}
}\sigma_{k}=\left\{
\begin{array}
[c]{c}%
0\text{ \textit{if} }W^{k}(x)\in\lbrack0,\rho],\\
1\text{ \textit{otherwise}.}%
\end{array}
\right.
\]
For example $\tau(0)=\overline{0}$ and $\tau(1)=\overline{1}$.

We need to understand the structure of $\overline{\Omega}$ because we will use
Proposition \ref{basicconjugacythm} to prove Theorem \ref{theorem1}. Since
$\Omega\subset I^{\infty}$ is totally disconnected and $A=[0,1]$ is connected,
it follows that $\tau:A\rightarrow\Omega$ is not a homeomorphism and hence
that $\overline{\Omega}\neq\Omega$. (Note that if $\overline{\Omega}=\Omega$
then $\tau:A\rightarrow\Omega$ is a homeomorphism, \cite[Theorem 3.2 (v)]{BHI}.)

To help to describe $\overline{\Omega}$ we introduce the mask $\mathcal{M}%
_{\rho}^{+}=$ $\{M_{0}^{+}=[0,\rho)$,$M_{1}^{+}=[\rho,1]\}$ for $\mathcal{F}$.
Let $\Omega_{+}$be the address space associated with the mask $\mathcal{M}%
_{\rho}^{+}$, and let $\tau^{+}:[0,1]\rightarrow\Omega_{+}$ be the
corresponding section. The corresponding masked dynamical system $W_{+}$ is
obtained by replacing $[0,\rho]$ by $[0,\rho)\ $in (\ref{Wequation}).

Proposition \ref{mihalachethm} is a summary of some results in
\cite{mihalache}, that concern the spaces $\Omega,\Omega_{+},\overline{\Omega
}$ and the sections $\tau$ and $\tau^{+}$. In particular, it describes the
monotonicity of $\tau$ and the subtle relationship between $\tau$ and
$\tau^{+}$, and it provides a characterization of $\overline{\Omega}$ in terms
of two itineraries.

\begin{proposition}
[\cite{mihalache}]\label{mihalachethm} For all $\rho\in\lbrack1-b,a],$

(i) $\Omega$ is closed from the left, $\Omega_{+}$ is closed from the right,
and
\[
\overline{\Omega}=\overline{\Omega_{+}}=\Omega\cup\Omega_{+}=\overline
{\Omega\cap\Omega_{+}};
\]

(ii) $\mathcal{\pi}^{-1}(x)\cap\overline{\Omega}=\{\tau(x),\tau^{+}(x)\}$ for
all $x\in$ $[0,1]$;

(iii) for all $x,y\in\lbrack0,1]$ with $x<y$,
\[
\tau(x)\preceq\tau^{+}(x)\prec\tau(y)\preceq\tau^{+}(y);
\]

(iv) for all $x\in\lbrack0,1],$ $\tau(x)\neq\tau^{+}(x)$ if only if $x\in
W^{-k}(\rho)$ for some $k\mathbb{\in N}$;

(v) $\tau(W(x))=S(\tau(x))$ and $\tau^{+}(W_{+}(x))=S(\tau^{+}(x))$ for all
$x\in\lbrack0,1]$;

(vi) $S(\Omega)=\Omega;$ $S(\Omega_{+})=\Omega_{+};$ $S(\overline{\Omega
})=\overline{\Omega};$

(vii) let $\tau(\rho)=\alpha$ and $\tau^{+}(\rho)=\beta$,%
\begin{align*}
\Omega &  =\{\sigma\in I^{\infty}:S^{k}(\sigma)\in\lbrack\overline{0}%
,\alpha]\cup(\beta,\overline{1}]\text{ for all }k\mathbb{\in N\}}\text{;}\\
\Omega_{+}  &  =\{\sigma\in I^{\infty}:S^{k}(\sigma)\in\lbrack\overline
{0},\alpha)\cup\lbrack\beta,\overline{1}]\text{ for all }k\mathbb{\in
N\}}\text{;}\\
\overline{\Omega}  &  =\{\sigma\in I^{\infty}:S^{k}(\sigma)\in\lbrack
\overline{0},\alpha]\cup\lbrack\beta,\overline{1}]\text{ for all }k\mathbb{\in
N\}}.
\end{align*}

\end{proposition}

\begin{proof}
Proof of (i): This is \cite[Proposition 2]{mihalache}.

Proof of (ii): From the definition of the section $\tau$ we have
$\tau(x)=\mathcal{\pi}^{-1}(x)\cap\Omega$. Similarly $\tau^{+}(x)=\mathcal{\pi
}^{-1}(x)\cap\Omega_{+}.$ The result now follows at once from (i)

Proof of (iii) and (iv): These are equivalent to \cite[section 2, (5) and
(6)]{mihalache}.

Proof of (v) and (vi): These are the content of \cite[Proposition
1]{mihalache}$.$

Proof of (vii): This follows from \cite[Proposition 3]{mihalache}.
\end{proof}

Let $\mathfrak{F}$ denote the set of all iterated function systems of the form
of $\mathcal{F}$ described above, at the start of Section
\ref{generalstructuresec}. Let$\ \widetilde{\mathcal{F}}\in\mathfrak{F}$, and
let corresponding quantities be denoted by tildas. That is, let
\[
\widetilde{\mathcal{F}}=([0,1]\subset\mathbb{R};\widetilde{f}_{0}%
(x),\widetilde{f}_{1}(x))
\]
where $\widetilde{f}_{i}(i)=i,0<\widetilde{f}_{i}(y)-\widetilde{f}%
_{i}(x)<\widetilde{\lambda}(y-x)$ for all $x<y$, $(i=0,1)\ $where
$\widetilde{f}_{0}(1)=\widetilde{a}\geq1-\widetilde{b}=\widetilde{f}_{1}(0)$
with $0<\widetilde{a},\widetilde{b}<1$. Let $\mathcal{M}_{\widetilde{\rho}%
}=\{[0,\widetilde{\rho}],(\widetilde{\rho},1]\}$ for $\widetilde{\rho}%
\in\lbrack1-\widetilde{b},\widetilde{a}]$ be a family of masks for
$\widetilde{\mathcal{F}},$ analogous to the masks $\mathcal{M}_{\rho}$ for
$\mathcal{F}$, and let $\widetilde{\tau}$ and $\widetilde{\tau}^{+}$ be the
corresponding sections for $\widetilde{\mathcal{F}}$, analogous to the
sections $\tau$ and $\tau^{+}$ for $\mathcal{F}$. Let $\widetilde{W}%
:[0,1]\rightarrow\lbrack0,1]$ denote the masked dynamical system for
$\widetilde{\mathcal{F}}$ corresponding to the mask $\mathcal{M}%
_{\widetilde{\rho}}$.

\begin{corollary}
\label{corollary1}The following statements are equivalent:

(i) the fractal transformation $T_{\mathcal{F}\widetilde{\mathcal{F}}%
}=\widetilde{\pi}\circ\tau:[0,1]\rightarrow\lbrack0,1]$ is an orientation
preserving homeomorphism;

(ii) $\widetilde{\tau}(\widetilde{\rho})=\tau(\rho)$ and $\widetilde{\tau}%
^{+}(\widetilde{\rho})=\tau^{+}(\rho);$

(iii) the masked dynamical systems $W:[0,1]\rightarrow\lbrack0,1]$ and
$\widetilde{W}:[0,1]\rightarrow\lbrack0,1]$ are topologically conjugate under
an orientation preserving homeomorphism.

\begin{proof}
Proof that (iii)$\Rightarrow$(ii): Let the homeomophism be $H:[0,1]\rightarrow
\lbrack0,1]$, such that $H(1)=1,$ so that $W=H^{-1}\widetilde{W}H$. Both
systems have the same set of itineraries and $x=H(\rho)=\widetilde{\rho}$ is
the location of the discontinuity of $\widetilde{W}$. The two itineraries
associated with the discontinuity must be the same for both systems, and in
the same order, because the homeomorphism is order preserving.

Proof that (ii)$\Rightarrow$(iii): By Proposition \ref{mihalachethm} (vii) the
closure of the address spaces for the two systems is the same, namely
$\overline{\Omega}$. We are going to use Proposition \ref{basicconjugacythm}
(iv), so we need to check the condition in Proposition \ref{basicconjugacythm}
(iii). Suppose that $\sigma,\omega\in$ $\overline{\Omega}$ and suppose that
$\pi(\sigma)=\pi(\omega)$. We need to show that $\widetilde{\pi}%
(\sigma)=\widetilde{\pi}(\omega)$. Since by Proposition
\ref{basicconjugacythm} $\overline{\Omega}=\Omega\cup\Omega_{+}$ we must have
either $\sigma,\omega\in\Omega$ or $\sigma,\omega\in\Omega_{+}$ or, without
loss of generality, $\sigma\in\Omega$ and $\omega\in\Omega_{+}$. If
$\sigma,\omega\in\Omega$ then $\pi(\sigma)=\pi(\omega)$ implies $\sigma
=\tau\circ\pi(\sigma)=\tau\circ\pi(\omega)=\omega,$ so $\sigma=\omega$, whence
$\widetilde{\pi}(\sigma)=$ $\widetilde{\pi}(\omega).$ Similarly, if
$\sigma,\omega\in\Omega_{+}$ then also $\widetilde{\pi}(\sigma)=$
$\widetilde{\pi}(\omega)$, but this time use $\tau_{+}$.

Now suppose $\sigma\in\Omega$, $\omega\in\Omega_{+},$ and $\pi(\sigma
)=\pi(\omega)$. Then we have $\widetilde{\pi}(\sigma)=$ $\widetilde{\pi}%
\circ\tau\circ\pi(\sigma)$ and $\widetilde{\pi}(\omega)=\widetilde{\pi}%
\circ\tau_{+}\circ\pi(\omega)=$ $\widetilde{\pi}\circ\tau_{+}\circ\pi(\sigma
)$. Again, if $\tau\circ\pi(\sigma)=\tau_{+}\circ\pi(\omega)(=\tau_{+}\circ
\pi(\sigma))$ we have $\widetilde{\pi}(\sigma)=$ $\widetilde{\pi}(\omega)$. So
we suppose $\tau\circ\pi(\sigma)\neq\tau_{+}\circ\pi(\omega)(=\tau_{+}\circ
\pi(\sigma)).$ But by Proposition \ref{mihalachethm}, $\tau\circ\pi
(\sigma)\neq\tau_{+}\circ\pi(\sigma)$ iff $\pi(\sigma)(=\pi(\omega))\in
W^{-k}(\rho)$. But $\pi(\sigma)\in W^{-k}(\rho)$ implies $\tau\circ\pi
(\sigma)=\sigma_{0}\sigma_{1}...\sigma_{k-1}\tau(\rho)$ for some sigmas and
$\tau_{+}\circ\pi(\sigma)=\tau_{+}\circ\pi(\omega)=\sigma_{0}\sigma
_{1}...\sigma_{k-1}\tau_{+}(\rho)$ whence $\widetilde{\pi}(\omega
)=\widetilde{\pi}\circ\tau_{+}\circ\pi(\omega)=\widetilde{f}_{\sigma_{0}%
\sigma_{1}...\sigma_{k-1}}\circ\widetilde{\pi}\circ\tau_{+}(\rho)$ and
$\widetilde{\pi}(\sigma)=\widetilde{\pi}\circ\tau\circ\pi(\sigma
)=\widetilde{f}_{\sigma_{0}\sigma_{1}...\sigma_{k-1}}\circ\widetilde{\pi}%
\circ\tau(\rho);$ but by (ii)\ these latter two quantities are the same. Hence
the condition in \ref{basicconjugacythm} (iv) is satisfied.

The other direction of the last part here is essentially the same, but swap
the roles of tildas with non-tildas.

Proof that (i) and (iii) are equivalent is similar to the above, again using
Proposition \ref{basicconjugacythm}.
\end{proof}
\end{corollary}

We conclude this section by outlining a direct proof of Theorem
\ref{directtheorem}. Let $\mathcal{F}$ and $\mathcal{G}$ are two overlapping
IFSs, as discussed in this paper, and let corresponding masked dynamical
systems (see Figure \ref{figaug17}) have address spaces $\Omega_{\mathcal{F}}$
and $\Omega_{\mathcal{G}},$ respectively. Here $\tau_{\mathcal{F}%
}:[0,1]\rightarrow\Omega_{\mathcal{F}}$ is the masked section for
$\mathcal{F}$, equal to the mapping from $x$ to the itinerary of $x,$ and
$\pi_{\mathcal{F}}:\Omega_{\mathcal{F}}\rightarrow\lbrack0,1]$ is the
(continuous, onto, restriction to $\Omega_{\mathcal{F}}$ of the) coding map
for $\mathcal{F}.$ Here too, $\tau_{\mathcal{G}}:[0,1]\rightarrow
\Omega_{\mathcal{G}}$ is the masked section for $\mathcal{G}$, equal to the
mapping from $x$ to the itinerary of $x,$ and $\pi_{\mathcal{G}}%
:\Omega_{\mathcal{G}}\rightarrow\lbrack0,1]$ is the (continuous, onto,
restriction to $\Omega_{\mathcal{G}}$ of the) coding map for $\mathcal{G}.$
Note that $\pi_{\mathcal{F}}=\tau_{\mathcal{F}}^{-1}$ and $\pi_{\mathcal{G}%
}=\tau_{\mathcal{G}}^{-1}$.

Let $\overline{\pi}_{\mathcal{F}}:\overline{\Omega_{\mathcal{F}}}%
\rightarrow\lbrack0,1]$ and $\overline{\pi}_{\mathcal{G}}:\overline
{\Omega_{\mathcal{G}}}\rightarrow\lbrack0,1]$ be the unique continuous
extensions of $\pi_{\mathcal{F}}$ and $\pi_{\mathcal{G}}$ respectively to the
closures of their domains. (They are the same as the restrictions of the
original coding maps (that act on $I^{\infty}$) to the closures of the
respective masked address spaces.)

\begin{theorem}
\label{directtheorem}Let $\Omega_{\mathcal{F}}=\Omega_{\mathcal{G}}$. Then
$\pi_{\mathcal{G}}\tau_{\mathcal{F}}:[0,1]\rightarrow\lbrack0,1]$ is a
homeomorphism. (Its inverse is $\pi_{\mathcal{F}}\tau_{\mathcal{G}}$.)
\end{theorem}

\begin{proof}
Let $\Omega=\Omega_{\mathcal{F}}=\Omega_{\mathcal{G}}$. Consider the mapping
$\pi_{\mathcal{G}}\tau_{\mathcal{F}}\overline{\pi}_{\mathcal{F}}%
:\overline{\Omega}\rightarrow\lbrack0,1]$. It is observed that $\tau
_{\mathcal{F}}\overline{\pi}_{\mathcal{F}}\sigma=\sigma$ for all $\sigma
\in\Omega$ so $\pi_{\mathcal{G}}\tau_{\mathcal{F}}\overline{\pi}_{\mathcal{F}%
}\sigma=\pi_{\mathcal{G}}\sigma$ is continuous at all points $\sigma\in
\Omega.$ In fact, since $\pi_{\mathcal{G}}$ is uniformly continuous it follows
that $(\pi_{\mathcal{G}}\tau_{\mathcal{F}}\overline{\pi}_{\mathcal{F}%
})|_{\Omega}:\Omega\rightarrow\lbrack0,1]$ is uniformly continuous. Hence it
has a unique continuous extension to $\overline{\Omega},$ namely
$\overline{\pi}_{\mathcal{G}}.$

We want to show that this continuous extension is the same as $\pi
_{\mathcal{G}}\tau_{\mathcal{F}}\overline{\pi}_{\mathcal{F}}$. This will show
that $\pi_{\mathcal{G}}\tau_{\mathcal{F}}\overline{\pi}_{\mathcal{F}%
}:\overline{\Omega}\rightarrow\lbrack0,1]$ is continuous.

Suppose $\overline{\sigma}\in\overline{\Omega}\backslash\Omega$. Then, by what
we know about the masked code spaces in question, there is a second point
$\sigma\in\Omega$ such that $\pi_{\mathcal{G}}\sigma=$ $\overline{\pi
}_{\mathcal{G}}\overline{\sigma}$ and $\pi_{\mathcal{F}}\sigma=$
$\overline{\pi}_{\mathcal{F}}\overline{\sigma}$. Hence we have $\pi
_{\mathcal{G}}\tau_{\mathcal{F}}\overline{\pi}_{\mathcal{F}}\overline{\sigma
}=\pi_{\mathcal{G}}\tau_{\mathcal{F}}\pi_{\mathcal{F}}\sigma=\pi_{\mathcal{G}%
}\sigma=\overline{\pi}_{\mathcal{G}}\overline{\sigma},$ which is what we
wanted to show.

We conclude this first part of the proof with the conclusion that
$\pi_{\mathcal{G}}\tau_{\mathcal{F}}\overline{\pi}_{\mathcal{F}}%
:\overline{\Omega}\rightarrow\lbrack0,1]$ is continuous.

Now use that facts that (1) $\pi_{\mathcal{G}}\tau_{\mathcal{F}}\overline{\pi
}_{\mathcal{F}}:\overline{\Omega}\rightarrow\lbrack0,1]$ is continuous and (2)
$\overline{\pi}_{\mathcal{F}}:\overline{\Omega}\rightarrow\lbrack0,1]$ is
continuous to conclude, by a well-known theorem (i.e. the one cited in the
Monthly), that $\pi_{\mathcal{G}}\tau_{\mathcal{F}}:[0,1]\rightarrow
\lbrack0,1]$ is continuous.

Similarly prove that $\pi_{\mathcal{F}}\tau_{\mathcal{G}}:[0,1]\rightarrow
\lbrack0,1]$ is continuous.\ Finally check that $\pi_{\mathcal{F}}%
\tau_{\mathcal{G}}\pi_{\mathcal{G}}\tau_{\mathcal{F}}=i_{[0,1]}$.
\end{proof}

\subsection{\label{trapsec}The special structure of the trapping region}

The material in this section is not needed towards the proof Theorems
\ref{theorem1} and \ref{theorem2ndpart}, but is of independent interest
because it concludes with Theorem \ref{shifttheorem} which connects the
present work to general binary symbolic dynamical systems.

Consider $W:[0,1]\rightarrow\lbrack0,1]$ as at the start of Section
\ref{introsec}. It is readily established that $W(0)=0$, $W(1)=1,$ and, for
any given $x\in(0,1),$ there exists an integer $K$ so that
\[
W^{k}(x)\in(W_{1}(\rho),W_{0}(\rho)]\text{ for all }k>K.
\]
Similarly for $W^{+}:[0,1]\rightarrow\lbrack0,1]$ we have $W^{+}(0)=0$,
$W^{+}(1)=1,$ and, for any given $x\in(0,1),$ there exists an integer $K$ so
that
\[
(W^{+})^{k}(x)\in\lbrack W_{1}(\rho),W_{0}(\rho))\text{ for all }k>K.
\]
We call $D=$ $[W_{1}(\rho),W_{0}(\rho)]$ the \textit{trapping region}. It is
readily checked that both $W|_{D}:D\rightarrow D$ and $W^{+}|_{D}:D\rightarrow
D$ are topologically transitive, and both are sensitively dependent on initial conditions.

The sets $D$ and $\Omega$ can be described in terms of $\tau(D)$. There is
some redundancy among the statements in Proposition \ref{trapprop}, but all
are informative. We define $s_{i}:I^{\infty}\rightarrow I^{\infty}$ by
$s_{i}(\sigma_{0}\sigma_{1}\sigma_{2}...)=i\sigma_{0}\sigma_{1}\sigma
_{2}...(i=0,1).$

\begin{proposition}
[\cite{mihalache}]\label{trapprop} Let $\tau(\rho)=\alpha$ and $\tau^{+}%
(\rho)=\beta$.

(i) $\overline{\tau(D)}=\tau(D)\cup\tau^{+}(D);$

(ii) $S(\tau(D))=\tau(D),S(\tau^{+}(D))=\tau^{+}(D),$ and $S(\overline
{\tau(D)})=\overline{\tau(D)}$;

(iii) $\alpha=01\alpha_{2}\alpha_{3}...;$ $\beta=10\beta_{2}\beta_{3}...$;

(iv) both $S^{k}(\alpha)\in\lbrack\overline{0},\alpha]\cup\lbrack
\beta,\overline{1}]$ and $S^{k}(\beta)\in\lbrack\overline{0},\alpha
]\cup\lbrack\beta,\overline{1}]$ for all $k\in\mathbb{N};$

(v) $\overline{\tau(D)}=\bigcap\limits_{k=0}^{\infty}F^{k}([S\left(
\beta\right)  ,S\left(  \alpha\right)  ])$ where $F:\mathbb{H(}I^{\infty
})\rightarrow\mathbb{H(}I^{\infty})$ is defined by%
\[
F(\Lambda)=s_{0}([S^{2}(\beta),S(\alpha)]\cap\Lambda)\cup s_{1}([S(\beta
),S^{2}(\alpha)]\cap\Lambda);
\]
(vi) $\overline{\tau(D)}=\bigcap\limits_{k=0}^{\infty}F^{k}([S(\beta
),\alpha]\cup\lbrack\beta,S(\alpha)])$;

(vii) $\overline{\Omega}=\{\sigma\omega:\sigma\in\{0\}^{k}\cup\{1\}^{k}%
,k=\mathbb{N}$,$\omega\in\overline{\tau(D)}\}$.
\end{proposition}

\begin{proof}
These statements are all direct consequences of \cite[Proposition
3]{mihalache} and the fact that the orbits of $\alpha$ and $\beta$ under
$S|_{\overline{\Omega}}:\overline{\Omega}\rightarrow\overline{\Omega}$ must
actually remain in $\overline{\tau(D)}$, the closure of the set of addresses
of the points in the trapping region. Of particular importance are (v), (vi)
and (vii) which taken together provide a detailed description of
$\overline{\Omega}$.
\end{proof}

The following Theorem provides characterizing information about \textit{all
}shift invariant subspaces of $I^{\infty}$. This is remarkable:\ it implies,
for example if $S\alpha\succ S^{2}\alpha\succ\beta\succ\alpha\succ$
$S^{2}\beta\succ S\beta$ then the overarching set contains a trapping region,
and submits to the description in (v), (vi) and (vii).

\begin{theorem}
\label{shifttheorem} Let $\Xi\subset I^{\infty}$ be shift foward invariant,
and let the following quantities be well defined:
\begin{align*}
\beta &  =\inf\{\sigma\in\Xi:\sigma_{0}=1\},\\
\alpha &  =\sup\{\sigma\in\Xi:\sigma_{0}=0\}.
\end{align*}
If $S\alpha\succ\beta$ and $\alpha\succ S\beta$ (the cases where one or other
of these two conditions does not hold are very simple and readily analyzed),
then $\overline{\Xi}\subset\overline{\Omega}$ where $\overline{\Omega}$ is
defined by Proposition \ref{trapprop} (v) and (vii). In particular, if
$S\alpha\succ S^{2}\alpha\succ\beta\succ\alpha\succ$ $S^{2}\beta\succ S\beta$
then the orbits of all points, except $\overline{0}$ and $\overline{1}$ end
up, after finitely many steps, in the associated trapping region defined in
Proposition \ref{trapprop} (vi).
\end{theorem}

\begin{proof}
Begin by noting that, by definition of $\alpha$ and $\beta$ we must have
$\overline{\Xi}\subset\lbrack\overline{0},\alpha]\cup\lbrack\beta,\overline
{1}].$ Since $S\overline{\Xi}=\overline{\Xi}$ we must have $S^{k}\sigma
\in\lbrack\overline{0},\alpha]\cup\lbrack\beta,\overline{1}]$ for all
$k\in\mathbb{N},$ for all $\sigma\in\Xi$. In particular, both $S^{k}%
(\alpha)\in\lbrack\overline{0},\alpha]\cup\lbrack\beta,\overline{1}]$ and
$S^{k}(\beta)\in\lbrack\overline{0},\alpha]\cup\lbrack\beta,\overline{1}]$ for
all $k\in\mathbb{N}.$ The result now follows by taking inverse limits as in
the first part of the proof of Proposition 3 in \cite{mihalache}.
\end{proof}

\begin{example}
It is well known that any dynamical system $f:X\rightarrow X$ can be
represented coarsely, by choosing as subset $O\subset X$ and defining
itineraries $\sigma_{k}=0$ if $f^{k}(x)\in O$, $=1$ otherwise. The resulting
set of itineraries is foward shift invariant, so Theorem \ref{shifttheorem}
applies. In many cases one can impose conditions on the original dynamical
system, such as a topology on $X$ and continuity, to cause Theorem
\textit{\ref{shifttheorem} to yield much stronger conclusions, such as
Sharkovskii's Theorem, and results of Milnor-Thurston. But even as it stands,
the theorem is very powerful, as the following simple example, which I have
not been able to find in the literature, demonstrates. Note that it is a kind
of min/max theorem.}
\end{example}

\begin{example}
Let $f:\{1,2,3,...\}\rightarrow\{1,2,3,...\}\ni x\mapsto x+1$ be a dynamical
system. Let $O\subset\{1,2,3,...\}$ be the set of integers that are not
primes, $O=\{1,4,6,8,9,10,....\}$. Let $\{f^{n}(x)\}_{n=0}^{\infty}$ denote
the orbit of $x$ and let $\tau(x)=\sigma\in I^{\infty}$ denote the
corresponding symbolic orbit defined by $\sigma_{k}=0$ if $f^{k}(x)\in O$,
$=1$ otherwise. Let $\Omega=\tau(\{1,2,3,...\})$. Then $S\Omega\subset\Omega$
so we can apply Theorem \ref{shifttheorem}. We readily find that
$\alpha=011010100...$ and $\beta=1000...$. It follows that
\[
\pi_{z}(\alpha)=(1-z)\sum\limits_{P\text{=Prime}}z^{P-1}\text{ and }\pi
_{z}(\beta)=(1-z)
\]
whence the corresponding piecewise linear dynamical system is $L$ with slope
$1/z$ and $p=(1-z)$ where $z$ is the positive solution of
\[
z=\sum\limits_{P\text{=Prime}}z^{P}\text{.}%
\]
It is readily verified, with the help of Theorem 1 that the dynamical system
$L=L_{\gamma=1/z,p=(1-z)}$ has this property: for all $n=0,1,2,...$we have
$L^{n}(1-z)>(1-z)$ iff $n+1$ is prime.
\end{example}

\begin{example}
This result also holds numerically, approximately: using Maple we find that
the solution to $\sum\limits_{P\text{=Prime,}P\leq29}z^{P}=z$ is approximately
$\tilde{z}=0.55785866$. Then, setting $\tilde{L}=L_{\gamma=1/\tilde
{z},p=(1-\tilde{z})}$, we find numerically that for $n\leq19,$ $\tilde{L}%
^{n}(1-\tilde{z})\in(1-\tilde{z},1]$ iff $n+1$ is prime.
\end{example}

\subsection{Affine systems}

Here we consider the case where $W$ is piecewise affine and, in particular, of
the form of $L$, on the right in Figure \ref{figaug17}. Let
\[
0<a,b<1,1\leq a+b,
\]%
\[
\mathcal{F}=\mathcal{F}(a,b)=([0,1]\subset\mathbb{R};f_{1}(x)=ax,f_{2}%
(x)=bx+(1-b)).
\]
The attractor of $\mathcal{F}$ is $A=[0,1]\subset\mathbb{R}$, and the coding
map is $\pi:I^{\infty}\rightarrow\lbrack0,1].$ Define a one parameter family
of masks
\[
\{\mathcal{M}_{\rho}:0<1-b\leq\rho\leq a<1\}
\]
for $\mathcal{F}$ by $M_{0}=[0,\rho],$ $M_{1}=(\rho,1]$. The corresponding
masked dynamical system is
\[
W:[0,1]\rightarrow\lbrack0,1],x\mapsto\left\{
\begin{array}
[c]{c}%
x/a\text{ if }x\in\lbrack0,\rho],\\
x/b+(1-1/b)\text{ otherwise}.
\end{array}
\right.
\]
The set of allowed parameters $\mathcal{(}a,b,\rho)$ is defined by the
simplex
\[
\mathcal{P=}\{\mathcal{(}a,b,\rho)\in\mathbb{R}^{3}:0<a,b<1,1\leq
a+b,1-b\leq\rho<a\}.
\]
We either suppress or reference $\mathcal{(}a,b,\rho)\in\mathcal{P}$ depending
on the context. The section is $\tau=\tau(a,b,\rho)$, the address space is
$\Omega=\Omega(a,b,\rho)=\tau(a,b,\rho)([0,1])=\tau([0,1]),$ and the masked
dynamical system is $W=W(a,b,\rho)$. We will denote the coding map for
$\mathcal{F}(a,b)$ by $\pi_{a,b}:I^{\infty}\rightarrow\lbrack0,1]$. In the
case where $a=b=\varsigma$ we will write this coding map as $\pi_{\varsigma
}:I^{\infty}\rightarrow\lbrack0,1]$.

The affine case is the key to Theorem 1, because we can write down explicitly
the mapping $\pi_{a,b}$ using the following explicit formulas:%

\begin{align*}
f_{i}(x)  &  =a^{1-i}b^{i}x+i(1-b)\text{ }(i=0,1)\\
f_{\sigma_{i}}(x)  &  =a^{1-\sigma_{i}}b^{\sigma_{i}}x+\sigma_{i}(1-b)\text{
}\\
f_{\sigma_{0}}(f_{\sigma_{1}}(x))  &  =a^{1-\sigma_{0}}b^{\sigma_{0}%
}a^{1-\sigma_{1}}b^{\sigma_{1}}x+(\sigma_{1}a^{1-\sigma_{0}}b^{\sigma_{0}%
}+\sigma_{0})(1-b)\\
&  =a^{2-\sigma_{0}-\sigma_{1}}b^{\sigma_{0}+\sigma_{1}}x+(a^{1-\sigma_{0}%
}b^{\sigma_{0}}\sigma_{1}+\sigma_{0})(1-b)\\
f_{\sigma_{0}}(f_{\sigma_{1}}(f_{\sigma_{2}}(x)))  &  =a^{3-\sigma_{0}%
-\sigma_{1}-\sigma_{2}}b^{\sigma_{0}+\sigma_{1}+\sigma_{2}}+(a^{1-\sigma
_{0}-\sigma_{1}}b^{\sigma_{0}+\sigma_{1}}\sigma_{2}+a^{1-\sigma_{0}}%
b^{\sigma_{0}}\sigma_{1}+\sigma_{0})(1-b),\\
&  \text{\textit{and so on....}}%
\end{align*}
We deduce%
\[
\pi_{a,b}(\sigma)=(1-b)\sum\limits_{k=0}^{\infty}a^{k-\sigma_{0}-\sigma
_{1}-...\sigma_{k-1}}b^{\sigma_{0}+\sigma_{1}+...\sigma_{k-1}}\sigma
_{k}\text{.}%
\]
Clearly, for each fixed $\sigma\in I^{\infty},$ $\pi_{a,b}(\sigma)$ is
analytic in $(a,b)\in\mathbb{C}^{2}$ with $|a|<1$ and $|b|<1.$ Also, since
\[
\left\vert \sum\limits_{k=0}^{\infty}a^{k-\sigma_{0}-\sigma_{1}-...\sigma
_{k-1}}b^{\sigma_{0}+\sigma_{1}+...\sigma_{k-1}}\sigma_{k}\right\vert \leq
\sum\limits_{k=1}^{\infty}\left(  \max\{a,b\}\right)  ^{k}\sigma_{k}%
\]
we have that%
\[
\left\vert \pi(\sigma)(a,b)\right\vert \leq(1-\max\{a,b\})^{-1}(1-b).
\]

It follows that%
\[
\pi_{\varsigma}(\sigma)=(1-\varsigma)\sum\limits_{k=0}^{\infty}\sigma
_{k}\varsigma^{k}\text{.}%
\]
Let $\sigma\in I^{\infty}$ be fixed for now, and let $\varsigma\in\mathbb{C}$
with $\left\vert \varsigma\right\vert <1$. We have
\[
\left\vert \sum\limits_{k=0}^{\infty}\sigma_{k}\varsigma^{k}\right\vert
\leq\sum\limits_{k=0}^{\infty}\left\vert \varsigma\right\vert ^{k}%
\leq(1-\left\vert \varsigma\right\vert )^{-1}\text{.}%
\]
Hence $\pi_{\varsigma}(\sigma)$ is analytic for $\left\vert \varsigma
\right\vert <1$ and can be continued to a continuous function on $\left\vert
\varsigma\right\vert \leq1.$ In particular $\pi_{1}(\sigma)$ is well defined,
finite and real.

\section{Proof of\ Theorem \ref{theorem1}}

\begin{lemma}
\label{lemmaA0} If $\sigma,\omega\in\overline{\Omega}$ with $\sigma\prec
\omega$ then there is $\theta\in I^{k}$ for some $k\in\mathbb{N}$ such that
both $\theta\alpha$ and $\theta\beta$ belong to $\overline{\Omega}$ and
\[
\sigma\preceq\theta\alpha\prec\theta\beta\preceq\omega.
\]

\end{lemma}

\begin{proof}
Let $\sigma,\omega\in\overline{\Omega}$ with $\sigma\prec\omega$. Then there
is $n\in\mathbb{N}$ such that $\sigma|_{n}=\omega|_{n}$, $\sigma_{n}%
=0,\omega_{n}=1$. It follow that $S^{n}\sigma\preceq\alpha\prec\beta\preceq
S^{n}\omega$ whence, setting $\theta=\sigma|_{n}$ we have the stated result.
Since $S\beta\prec\alpha\prec\beta\prec S\alpha$, it follows from
\cite[Proposition 3]{mihalache} that $\theta\alpha\in\Omega,\theta\beta
\in\Omega_{+},$ and both $\theta\alpha$ and $\theta\beta\in\overline{\Omega}.$
\end{proof}

\begin{lemma}
\label{lemmaA} If there is $\gamma\in(1/3,1)$ such that $\gamma=\inf\{\zeta
\in(0,1):\pi_{\varsigma}\alpha=$ $\pi_{\varsigma}\beta\},$ and if
$\varsigma<\gamma$, $\sigma,\omega\in\overline{\Omega}$ and $\sigma\prec
\omega,$ then $\pi_{\varsigma}(\sigma)<\pi_{\varsigma}(\omega)$ and
$\pi_{\gamma}(\sigma)\leq\pi_{\gamma}(\omega)$. If there is no such $\gamma
\in(1/3,1)$, then $\sigma,\omega\in\overline{\Omega}$ and $\sigma\prec\omega,$
imply $\pi_{\varsigma}(\sigma)<\pi_{\varsigma}(\omega)$ for all $\varsigma<1$.
\end{lemma}

\begin{proof}
Suppose that is $\gamma\in(1/3,1)$ such that $\gamma=\inf\{\zeta\in
(0,1):\pi_{\varsigma}\alpha=$ $\pi_{\varsigma}\beta\}.$ Let
\[
\hat{\varsigma}=\inf\{\varsigma\in(0,1):\exists\sigma,\omega\in\overline
{\Omega},\text{ }\sigma_{0}=0,\omega_{0}=1,\text{\textit{s.t}. }\pi
_{\varsigma}(\sigma)=\pi_{\varsigma}(\omega)\}.
\]
Using continuity of $\pi_{\varsigma}\sigma$ in $\varsigma$ and $\sigma,$ and
compactness of $\overline{\Omega}$ it follows that there is $\sigma,\omega
\in\overline{\Omega},$ with $\sigma_{0}=0,\omega_{0}=1,$ such that
\[
\pi_{\hat{\varsigma}}(\sigma)=\pi_{\hat{\varsigma}}(\omega).
\]
Suppose $\hat{\varsigma}<\gamma$. So, for brevity we will assume that both
$\sigma\neq\alpha$ and $\beta\neq\omega$; it will be seen that the other two
possibilities can be handled similarly. We have that
\[
\pi_{1/3}(\sigma)\leq\pi_{1/3}(\alpha)<\pi_{1/3}(\beta)\preceq\pi_{1/3}%
(\omega)
\]
because $\pi_{1/3}$ is order preserving. We also have,
\[
\pi_{\varsigma}(\alpha)<\pi_{\varsigma}(\beta)\text{ for all }\varsigma
\leq\hat{\varsigma},
\]
Hence, by the intermediate value theorem, since $\pi_{\hat{\varsigma}}%
(\sigma)=\pi_{\hat{\varsigma}}(\omega)$, there is $\xi\prec\hat{\varsigma}$
such that either $\pi_{\xi}(\sigma)=\pi_{\xi}(\alpha)$ with $\sigma\neq\alpha$
or $\pi_{\xi}(\beta)=\pi_{\xi}(\omega)$ with $\beta\neq\omega.$ That is,
either the graph of $\pi_{\xi}(\sigma)$, as a function of $\xi$, intersects
the graph of $\pi_{\xi}(\alpha)$ at some point $\xi<\hat{\varsigma}$ or the
graph of $\pi_{\xi}(\beta)$ intersects the graph of $\pi_{\xi}(\omega)$ at
some point $\xi<\hat{\varsigma}$. Suppose that $\pi_{\xi}(\sigma)=\pi_{\xi
}(\alpha)$ with $\sigma\neq\alpha.$ Let $k$ be the least integer such that
$\pi_{\xi}(\sigma)=\pi_{\xi}(\alpha)$ with $(S^{k}\sigma)_{0}\neq(S^{k}%
\alpha)_{0}$ and let $\{\tilde{\sigma},\tilde{\omega}\}=\{S^{k}\sigma
,S^{k}\alpha\}$ where $\tilde{\sigma}_{0}=0$ and $\tilde{\omega}_{0}=1.$ Then
we note that $\pi_{\xi}(\sigma)=\pi_{\xi}(\alpha)$ implies $\pi_{\xi}%
(\tilde{\sigma})=\pi_{\xi}(\tilde{\omega}),$with $\tilde{\sigma}_{0}=0$ and
$\tilde{\omega}_{0}=1$. Hence $\xi\geq\hat{\varsigma}$ which is a
contradiction. Similarly we show that it cannot occur that $\pi_{\xi}%
(\beta)=\pi_{\xi}(\omega)$ with $\beta\neq\omega.$ We conclude that
$\hat{\varsigma}=\gamma.$

Hence if $\varsigma<\gamma$, $\sigma,\omega\in\overline{\Omega}$ and
$\sigma\prec\omega,$ then we cannot have $\pi_{\varsigma}(\sigma
)=\pi_{\varsigma}(\omega).$ (For if so, then, similarly to above, let $k$ be
the least integer such that $\pi_{\xi}(\sigma)=\pi_{\xi}(\omega)$ with
$(S^{k}\sigma)_{0}\neq(S^{k}\omega)_{0}$ and set $\{\tilde{\sigma}%
,\tilde{\omega}\}=\{S^{k}\sigma,S^{k}\omega\}$ where $\tilde{\sigma}_{0}=0$
and $\tilde{\omega}_{0}=1;$ then $\pi_{\varsigma}(\sigma)=\pi_{\varsigma
}(\omega)$ would imply $\pi_{\xi}(\tilde{\sigma})=\pi_{\xi}(\tilde{\omega}%
),$with $\tilde{\sigma}_{0}=0$ and $\tilde{\omega}_{0}=1$.) Since $\pi
_{1/3}(\sigma)<\pi_{1/3}(\omega)$ we conclude (again appealing to the
intermediate value theorem) that $\pi_{\varsigma}(\sigma)<\pi_{\varsigma
}(\omega)$ for all $\varsigma<\gamma$ and $\pi_{\gamma}(\sigma)\leq\pi
_{\gamma}(\omega)$. This completes the proof of the first part of the lemma.

Now suppose that there is no $\gamma\in(1/3,1)$ such that $\gamma=\inf
\{\zeta\in(0,1):\pi_{\varsigma}\alpha=$ $\pi_{\varsigma}\beta\}$. Then, if
$\sigma,\omega\in\overline{\Omega}$ and $\sigma\prec\omega,$ the fact that
$\pi_{1/3}(\sigma)<\pi_{1/3}(\omega)$ and the continuity of both
$\pi_{\varsigma}(\sigma)$ and $\pi_{\varsigma}(\omega)$ in $\varsigma,$ for
all $\varsigma<1,$ imply $\pi_{\varsigma}(\sigma)<\pi_{\varsigma}(\omega)$ for
all $\varsigma<1$.
\end{proof}

The following lemma sharpens the first sentence in Lemma \ref{lemmaA}. (Recall
our notation $\theta|_{k}=\theta_{0}\theta_{1}...\theta_{k-1}$.)

\begin{lemma}
\label{lemmastrict} If there is $\gamma\in(1/3,1)$ such that $\gamma
=\inf\{\zeta\in(0,1):\pi_{\varsigma}\alpha=$ $\pi_{\varsigma}\beta\},$
$\sigma,\omega\in\Omega$ and $\sigma\prec\omega,$ then $\pi_{\gamma}%
(\sigma)<\pi_{\gamma}(\omega)$.
\end{lemma}

\begin{proof}
Let $\sigma,\omega\in\Omega$ and $\sigma\prec\omega.$ Then we can find
$\theta\in\Omega$ such that $\sigma\prec\theta\prec\omega$. It follow from
\cite[(3)]{mihalache} and Proposition \ref{trapprop} we can find $k$ such that
$\sigma|_{k}\neq\theta|_{k}\neq\omega|_{k}$ and either $\{\theta|_{k}%
0\beta,\theta|_{k}\alpha\}\subset\overline{\Omega}$ or $\{\theta|_{k}%
\beta,\theta|_{k}1\alpha\}\subset\overline{\Omega}$. (This expresses the fact
that the $sup$ and $inf$ of each of the cylinder sets
\[
\left(  \theta_{0}\theta_{1}...\theta_{k-1}\right)  :=\{\eta\in\Omega
:\eta|_{k}=\theta|_{k}\}
\]
must belong to the inverse images of the critical point, together with
$\overline{0}$ and $\overline{1}$, see \cite[(3)]{mihalache}$.$ By taking $k$
sufficiently large the possibilities $\overline{0}$ or$\overline{\text{ }1}$
are ruled out by density of the set of all inverse images of the critical
point, see \cite[(4)]{mihalache}.)

If $\theta_{k}\theta_{k+1}=01$ then $\theta|_{k}0\beta\in\overline{\Omega}$,
$\theta|_{k}\alpha\in\Omega,$ and
\[
\sigma\prec\theta|_{k}0\beta\prec\theta|_{k}\alpha\prec\omega;
\]
and if $\theta_{k}\theta_{k+1}=10$ then $\theta|_{k}\beta\in\overline{\Omega}%
$, $\theta|_{k}1\alpha\in\Omega,$ and
\[
\sigma\prec\theta|_{k}\beta\prec\theta|_{k}1\alpha\prec\omega.
\]

Suppose $\theta_{k}\theta_{k+1}=01.$ It follows using Lemma \ref{lemmaA} that
\[
\pi_{\gamma}\sigma\leq\pi_{\gamma}\theta|_{k}0\beta\text{ and }\pi_{\gamma
}\theta|_{k}\alpha\leq\pi_{\gamma}\omega\text{.}%
\]
Now observe that
\begin{align*}
\pi_{\gamma}\theta|_{k}\alpha-\pi_{\gamma}\theta|_{k}0\beta &  =\gamma^{k}%
(\pi_{\gamma}\alpha-\pi_{\gamma}0\beta)\\
&  =\gamma^{k}(\pi_{\gamma}\alpha-\gamma\pi_{\gamma}\beta)\\
&  =\gamma^{k}(1-\gamma)\pi_{\gamma}\alpha>0.
\end{align*}

It follows that
\[
\pi_{\gamma}\sigma<\pi_{\gamma}\omega.
\]
A similar calculation deals with the case $\theta_{k}\theta_{k+1}=01$. The
possibility that there is no $k$ such that $\theta_{k}\theta_{k+1}%
\in\{01,10\}$ can be dealt with by related arguments.
\end{proof}

\begin{lemma}
\label{lemmaC} Let $\hat{\varsigma}\in(0,1]$ be maximal such that
$\pi_{\varsigma}\alpha<\pi_{\varsigma}\beta$ for all $\varsigma\in
(0,\hat{\varsigma}).$ Then
\[
(\pi_{\varsigma}\tau^{+}\rho-\pi_{\varsigma}\tau\rho)<(1-\varsigma
)/(1+\varsigma)
\]
for all $\varsigma\prec\hat{\varsigma}$.
\end{lemma}

\begin{proof}
Note the identity, for any $\sigma\in I^{\infty}$ and any $\varsigma\in(0,1),$%
\[
\pi_{\varsigma}\sigma=(1-\varsigma)\sigma_{0}+\varsigma\pi_{\varsigma}%
S\sigma\text{.}%
\]
Note too, by a simple calculation, $\tau^{+}\rho\prec S\tau\rho$ (i.e.
$\beta\prec S\alpha$) and $S\tau^{+}\rho\prec\tau\rho$ (i.e. $S\beta
\prec\alpha$)$.$ Hence, if $\varsigma\in(0,\hat{\varsigma})$, then by Lemma
\ref{lemmaA},
\begin{align*}
\pi_{\varsigma}\tau\rho &  =(1-\varsigma)(\tau\rho)_{0}+\varsigma
\pi_{\varsigma}S\tau\rho>(1-\varsigma)(\tau\rho)_{0}+\varsigma\pi_{\varsigma
}\tau^{+}\rho,\\
\pi_{\varsigma}\tau^{+}\rho &  =(1-\varsigma)(\tau^{+}\rho)_{0}+\varsigma
\pi_{\varsigma}S\tau^{+}\rho<(1-\varsigma)(\tau^{+}\rho)_{0}+\varsigma
\pi_{\varsigma}\tau\rho,
\end{align*}
whence, subtracting the first equation from the second, we get%
\begin{align*}
\pi_{\varsigma}\tau^{+}\rho-\pi_{\varsigma}\tau\rho &  <(1-\varsigma)(\tau
^{+}\rho)_{0}+\varsigma\pi_{\varsigma}\tau\rho-(1-\varsigma)(\tau\rho
)_{0}-\varsigma\pi_{\varsigma}\tau^{+}\rho\\
&  =(1-\varsigma)((\tau^{+}\rho)_{0}-(\tau\rho)_{0})-\varsigma(\pi_{\varsigma
}\tau^{+}\rho-\pi_{\varsigma}\tau\rho)
\end{align*}
which implies%
\[
(\pi_{\varsigma}\tau^{+}\rho-\pi_{\varsigma}\tau\rho)(1+\varsigma
)<(1-\varsigma)((\tau^{+}\rho)_{0}-(\tau\rho)_{0})
\]
whence%
\[
(\pi_{\varsigma}\tau^{+}\rho-\pi_{\varsigma}\tau\rho)<(1-\varsigma
)/(1+\varsigma).
\]

\end{proof}

It follows that either there exists $\gamma$ as described in Theorem
\ref{theorem1} with $\gamma<1$, or else $\pi_{\varsigma}\tau^{+}\rho
<\pi_{\varsigma}\tau\rho$ for all $\varsigma<1$ and $\lim_{\varsigma
\rightarrow1}(\pi_{\varsigma}\tau^{+}\rho-\pi_{\varsigma}\tau\rho)=0$.

Later we will show that the second option implies that the system must have
zero entropy, which is impossible, so the second option here cannot occur. But
first we explain why, in the first case, Theorem \ref{theorem1} is implied.

\begin{lemma}
\label{lemmaB} Let there exist $\gamma$ as described in Theorem \ref{theorem1}
and let $\gamma<1$. Let $p=\tau(\rho)(\gamma)(=\pi_{\gamma}\alpha)$ as in the
statement of Theorem \ref{theorem1}. Then the address space of $L_{\gamma
,p}:[0,1]\rightarrow\lbrack0,1]$ equals $\Omega.$
\end{lemma}

\begin{proof}
The address space of $L_{\gamma,p}$ is uniquely determined by the two
itineraries of $p=\pi_{\gamma}\tau\rho$. By direct calculation it is verified
that these addresses must be $\alpha=\tau\rho$ and $\beta=\tau^{+}\rho$.
Simply apply $L_{\gamma,p}$ to $\pi_{\gamma}\tau\rho$ and $L_{\gamma,p}^{+}$
to $\pi_{\gamma}\tau^{+}\rho$, and with the aid of Lemma \ref{lemmaA},
validate that the addresses of these two points, $\pi_{\gamma}\tau\rho$ and
$\pi_{\gamma}\tau^{+}\rho$ are indeed $\tau\rho$ and $\tau^{+}\rho$.
(\textbf{The assertion }$\sigma\prec\omega\Rightarrow$\textbf{ }$\pi_{\gamma
}(\sigma)<\pi_{\gamma}(\omega),$\textbf{ assured by Lemma \ref{lemmastrict},
implies that at each iterative step, in the computation of }$\pi_{\gamma}%
\tau\rho$\textbf{, the current point lies in the domain of the appropriate
branch of }$L_{\gamma,p}$\textbf{.})
\end{proof}

\begin{lemma}
\label{lemmaE} Let there exist $\gamma$ as described in Theorem \ref{theorem1}
and let $\gamma<1$. Then Theorem \ref{theorem1} holds.
\end{lemma}

\begin{proof}
This follows from Lemma \ref{lemmaB} together with Proposition
\ref{mihalachethm} parts (iii) and (iv). (Equivalently it follows from
Corollary \ref{corollary1}: I think I have in effect duplicated the intended
proof of Corollary \ref{corollary1}.)
\end{proof}

\begin{lemma}
\label{lemmaD} The equations
\[
\tau(\rho)(\gamma)=\tau(\rho+)(\gamma),\text{ with }\tau(\rho)(\varsigma
)<\tau(\rho+)(\varsigma)\text{ for }\varsigma\in\lbrack0.5,\gamma)\text{,}%
\]
possess a unique solution $\gamma<1$.
\end{lemma}

\begin{proof}
If $\pi_{\varsigma}\tau^{+}\rho<\pi_{\varsigma}\tau\rho$ for all $\varsigma<1$
and $\lim_{\varsigma\rightarrow1}(\pi_{\varsigma}\tau^{+}\rho-\pi_{\varsigma
}\tau\rho)=0$ then by (similar argument to the one in Lemma\ref{lemmaC}) that
$x$ implies $\pi_{\varsigma}\tau x<\pi_{\varsigma}\tau y$ for all
$\varsigma<1$. It follows that $\pi_{\varsigma}\tau([0,1])$ is an invariant
subset of the dynamical system $L_{p,\varsigma}$ for all $\varsigma<1,$ whence
the topological entropy of $W$ is less than $\log1/\lambda$ which is
impossible because $W$ has slope larger than or equal to $1/\lambda$.)
\end{proof}

Theorem \ref{theorem1} follows from Lemmas \ref{lemmaE} and \ref{lemmaD}.

\section{Proof of Theorem \ref{theorem2ndpart}}

This is straightfoward.

\begin{acknowledgement}
We thank Jed Keesling, Nina Snigreve, and Tony Samuels, for many helpful
discussions in connection with this work. We thank Konstantin Igudesman for
useful comments on an early version of this work.
\end{acknowledgement}

\end{document}